\documentclass[12pt]{amsart}

\usepackage{amssymb}

\newtheorem{theorem}{Theorem}[section]
\newtheorem{prop}[theorem]{Proposition}
\newtheorem{lemma}[theorem]{Lemma}
\newtheorem{corollary}[theorem]{Corollary}

\newtheorem{remark}[theorem]{Remark}

\newcommand{\E}{{\mathbb E}}
\newcommand{\F}{{\mathbb F}}
\newcommand{\Z}{{\mathbb Z}}
\newcommand{\sM}{{\mathcal M}}
\newcommand{\sH}{{\mathcal H}}
\newcommand{\sT}{{\mathcal T}}
\newcommand{\sN}{{\mathcal N}}
\newcommand{\sB}{{\mathcal B}}
\newcommand{\sE}{{\mathcal E}}

\newcommand{\Spec}{{\rm Spec \,}}
\newcommand{\Gal}{{\rm Gal}}
\newcommand{\Pic}{{\rm Pic}}
\newcommand{\G}{{\mathbb G}}
\newcommand{\sL}{{\mathcal L}}

\numberwithin{equation}{section}

\begin{document}

\baselineskip=15.5pt

\title[Brauer group of moduli spaces]{Brauer group of moduli
spaces of $\text{PGL}(r)$--bundles over a curve}

\author[I. Biswas]{Indranil Biswas}

\address{School of Mathematics, Tata Institute of Fundamental
Research, Homi Bhabha Road, Bombay 400005, India}

\email{indranil@math.tifr.res.in}

\author[A. Hogadi]{Amit Hogadi}

\address{School of Mathematics, Tata Institute of Fundamental
Research, Homi Bhabha Road, Bombay 400005, India}

\email{amit@math.tifr.res.in}



\date{}

\begin{abstract}
We compute the Brauer group of the moduli stack of stable 
$\text{PGL}(r)$--bundles
on a curve $X$ over an algebraically closed field of 
characteristic zero. We also show that this Brauer group of such
a moduli stack coincides with the Brauer group of the smooth locus
of the corresponding coarse moduli space of stable 
$\text{PGL}(r)$--bundles on $X$.
\end{abstract}

\maketitle

\section{Introduction}\label{sec1}

Let $k$ be an algebraically closed field of characteristic 
zero. Let $X$ be an irreducible smooth projective curve defined over 
$k$ of genus $g$, with $g\, \geq\, 2$. Fix an integer $r\geq 2$.
Let $\sN(r)$ be the moduli stack of stable $\text{PGL}(r,k)$ bundles 
on $X$. Let
\begin{equation}\label{c1}
\sN(r)\,\longrightarrow\, N(r)
\end{equation}
be the coarse moduli space.
Our aim is to study the Brauer group of $\sN(r)$ and that
of the smooth 
locus of $N(r)$.

Throughout this paper 
we assume that if $g\,= \,2$, then $r$ is at least three.

The following proposition is proved in Section \ref{sec2}
(see Proposition \ref{prop-n1}):

\begin{prop}\label{thm:cod2}
The morphism $\sN(r) \,\longrightarrow\, N(r)$ is an isomorphism
outside a codimension three closed subset. In particular, this
morphism identifies the Brauer group of the smooth locus of $N(r)$
with the Brauer group $Br(\sN(r))$.
\end{prop}

Let
\begin{equation}\label{nu0}
N^0(r)\, \subset\, N(r)
\end{equation}
be the locus of stable projective
bundles admitting no nontrivial automorphisms. Similarly, let
${\sN}^0(r)\, \subset\, \sN(r)$ be the substack of stable projective
bundles whose automorphism group is trivial. Then
the natural morphism ${\sN}^0(r)\,\longrightarrow\, N^0(r)$ is an
isomorphism.

Fix a line bundle $\xi$ over $X$.
Let ${\mathcal M}(r,\xi)$ denote the moduli stack of stable vector 
bundles $E\, \longrightarrow\, X$ of rank $r$ and determinant $\xi$,
meaning $\bigwedge^r E$ is isomorphic to $\xi$. This moduli stack
$\mathcal M(r,\xi)$ is an irreducible smooth
Deligne--Mumford stack having a 
smooth quasiprojective coarse moduli space. This coarse moduli space,
which is of dimension $(g-1)(r^2-1)$, will be
denoted by $M(r,\xi)$. The Brauer group of $M(r,\xi)$ has the
following description: the natural morphism
$$
\sM(r,\xi)\,\longrightarrow\, M(r,\xi)
$$
is a $\G_m$--gerbe, and its class in $Br(M(r,\xi))$ is a generator
of order ${\rm gcd}(r,\deg(\xi))$ \cite{BBGN}.

The connected components of both $\sN(r)$ and $N(r)$ are indexed by 
$\Z/r\Z$ 
$$
\sN(r)\, = \, \bigcup_{i\in \Z/r\Z} \sN(r)_i
\, ~\text{~\,and~\,}\, ~ N(r)\, = \, \bigcup_{i\in \Z/r\Z} N(r)_i\, .
$$
The component $N(r)_i$ corresponds to the projective bundles
associated to the vector bundles on $X$ of degree $i$. More
precisely, if ${\rm degree}(\xi)\,\equiv \,i$
mod $r$, then $N(r)_i$ is the quotient of
$M(r,\xi)$ by the subgroup 
$$
\Gamma\, \subset\, \Pic(X)
$$
of $r$--torsion points; the action of any $\zeta\, \in\, \Gamma$
on $M(r,\xi)$ sends any $E$ to $E\bigotimes\zeta$.

As before, fix a line bundle $\xi\, \longrightarrow\, X$. The image
of $\text{degree}(\xi)$ in ${\mathbb Z}/r{\mathbb Z}$ will be denoted
by $i$. If $g\,=\, 2$, then we assume that $r\, \geq\, 3$. For any 
$\zeta\, \in\, \Gamma$, let
$$
\widehat{\zeta}\, :\, M(r,\xi)\, \longrightarrow\, M(r,\xi)
$$
be the automorphism defined by $E\,\longmapsto\, E\bigotimes\zeta$.
Let
$$
{\mathcal L}_0\,\in\, \text{Pic}(M(r,\xi))\,=\, \mathbb Z
$$
be the ample generator.
Let $\widetilde{\Gamma}_i$ be the group of pairs of the form
$(\zeta\, ,\sigma)$, where $\zeta\, \in\, \Gamma$ and
$\sigma$ is an isomorphism of ${\mathcal L}_0$ with
$\widehat{\zeta}^*{\mathcal L}_0$. So we have a central
extension
$$
0\, \longrightarrow\, k^*\, \longrightarrow\, \widetilde{\Gamma}_i
\, \longrightarrow\,\Gamma \, \longrightarrow\, 0\, .
$$
Let $\nu_i\, \in\, H^2(\Gamma\, ,k^*)$ be the corresponding
extension class.

We prove the following theorem; its proof is given in
Section \ref{sec3}.

\begin{theorem}\label{thm:br}
There is a short exact sequence 
$$ 0 \,\longrightarrow\, \frac{H^2(\Gamma\, ,k^*)}{H} 
\,\longrightarrow\, Br(\sN(r)_i) \,
\stackrel{\tau}{\longrightarrow}\, \Z/\delta\Z \,
\longrightarrow\, 0\, , $$
where $H\,\subset\, H^2(\Gamma\, ,k^*)$ is the subgroup of order 
$\delta\,:=\, {\rm g.c.d.}(r,i)$ generated by $\nu_i$.
\end{theorem}

The class $\nu_i$ is described in Remark \ref{rem-ex}.

The short exact sequence in Theorem \ref{thm:br} is constructed
using the Leray spectral sequence for the quotient map
$M(r,\xi)\,\longrightarrow\, M(r,\xi)/\Gamma\,=\,N(r)_i$.

The following theorem provides more information on
the exact sequence in Theorem \ref{thm:br}.

\begin{theorem}\label{thm:order=r}
For a point $x \,\in\, X(k)$, let $i_x\,:\, \sN(r) \,\longrightarrow\,
\sN(r)\times_k X$ be the constant section of the projection
$\sN(r)\times_k X\, \longrightarrow\, \sN(r)$ passing through
$x$. Let $\alpha\,\in\, Br(\sN(r))$ be the class of the  
projective bundle obtained by pulling back the universal
${\rm PGL}(r,k)$ bundle over $\sN(r)\times_kX$ by $i_x$.

Then the following two hold:
\begin{enumerate}
 \item[(i)] The class $\alpha$ is independent of $x$, and $\tau(\alpha)
 \,=\, 1\,\in\, \Z/\delta\Z$, where $\tau$ is the homomorphism
 in Theorem \ref{thm:br}.
 \item[(ii)] The order of $\alpha$ restricted to each component of
$\sN(r)$ is exactly $r$. Therefore, the short exact sequence in
Theorem \ref{thm:br} is split if and only if $\delta\, =\, r$.
\end{enumerate}
\end{theorem}

{}From the second part of Theorem \ref{thm:order=r} it follows
that the principal $\text{PGL}(r,k)$--bundle over
$N^0(r)$ obtained by restricting the projective bundle
in Theorem \ref{thm:order=r} is stable (see Corollary \ref{corf}).

\medskip
\noindent
\textbf{Acknowledgements.}\, We thank the referee for comments to
improve the exposition.

\section{Fixed points of the moduli space}\label{sec2}

Throughout this section we fix a line bundle $\xi$ over
$X$, and we also fix an integer $r\, \geq\, 2$. As mentioned in the 
introduction, if $g\,=\, 2$,
then $r$ is taken to be at least three. Let $M(r,\xi)$ denote the
moduli space of stable vector bundles $E\, \longrightarrow\, X$
of rank $r$ with
$\bigwedge^r E\,=\, \xi$. For convenience the moduli
space $M(r,\xi)$ will also be denoted by $M$. Let 
\begin{equation}\label{e1}
\Gamma\, :=\, \text{Pic}^0(X)_r\, =\,\{\eta\,\in\,
\text{Pic}^0(X)\, \mid\, \eta^{\otimes r}\, \cong\, {\mathcal O}_X\}
\end{equation}
be the subgroup of $r$--torsion points of the Picard
group of $X$. For any $\eta\, 
\in\,
\Gamma$, we have the automorphism
\begin{equation}\label{e2}
\phi_\eta\, :\, {M}\, \longrightarrow\, M
\end{equation}
that sends any $E$ to $E\bigotimes \eta$. Let
\begin{equation}\label{e3}
\phi\, :\, \Gamma\, \longrightarrow\, \text{Aut}({M})
\end{equation}
be the homomorphism defined by $\eta\, \longmapsto\,\phi_\eta$.

Take any nontrivial line bundle
$$
\eta\, \in\, \Gamma\setminus\{{\mathcal O}_X\}\, .
$$
Let
\begin{equation}\label{e4}
M^\eta\, \subset\, M
\end{equation}
be the closed subvariety fixed by the automorphism
$\phi_\eta$
in \eqref{e2}. We will describe $M^\eta$.

Let $m$ be the order of $\eta$. Fix a nonzero
section
\begin{equation}\label{sigma}
\sigma\, :\, X\, \longrightarrow\, \eta^{\otimes m}\, .
\end{equation}
Consider the morphism of varieties
$$
f_\eta\, :\, \eta\, \longrightarrow\, \eta^{\otimes m}
$$
defined by $v\, \longmapsto\, v^{\otimes m}$. Set
\begin{equation}\label{e5}
Y\,:=\, f^{-1}_\eta(\sigma(X))\,\subset\, \eta\, ,
\end{equation}
where $\sigma$ is the section in \eqref{sigma}. Let
\begin{equation}\label{e6}
\gamma\, :\, Y\, \longrightarrow\, X
\end{equation}
be the morphism obtained by restricting
the natural projection of $\eta$ to $X$. So
$\gamma$ is an \'etale Galois covering
of degree $m$.

The curve $Y$ is irreducible. To prove this, let $Y_1$ be
an irreducible component of $Y$. Let $m_1$ be the degree of
the restriction $\gamma_1\, :=\, \gamma\vert_{Y_1}$, where $\gamma$
is defined in \eqref{e6}. We have a nonzero section
of $\eta^{\otimes m_1}$ whose fiber over any point $x\, \in\, X$
is $\bigotimes_{\theta\in \gamma^{-1}_1(x)} \theta$ (recall
that $\gamma^{-1}_1(x)$ is a subset of the fiber $\eta_x$).
Therefore, the line bundle $\eta^{\otimes m_1}$ is trivial.
Since the order of $\eta$ is precisely $m$, we have $Y_1\,=\, Y$.
Hence $Y$ is irreducible.

We note that the covering $Y$ is independent of the
choice of the section $\sigma$. Indeed, if $\sigma$ is replaced by
$c^m\cdot\sigma$, where $c\, \in\, k^*$, then the
automorphism of $\eta$ defined by multiplication with $c$ takes
$Y$ in \eqref{e5} to the covering corresponding to the
section $c^m\sigma$.

Let $N_\gamma$ denote the moduli space
of stable vector bundles $F$ over $Y$ of rank
$r/m$ with $\det \gamma_* F\, =\, \bigwedge^r \gamma_* F
\, \cong\, \xi$.

\begin{lemma}\label{lem1}
There is a nonempty Zariski open subset $U\,
\subset\, N_\gamma$ such that for
any $F\, \in\, U$,
$$
\gamma_*F\, \in\, M^\eta\, .
$$
Furthermore, the morphism
$$
U\, \longrightarrow\, M^\eta
$$
defined by $F\,\longmapsto\,\gamma_*F$ is surjective.
\end{lemma}

\begin{proof}
Note that the line bundle
$\gamma^*\eta$ has a tautological trivialization
because $Y$ is contained in the complement of the
image of the zero section in the total space of $\eta$
over which the pullback of $\eta$ is tautologically trivial.
Take any vector bundle $V$ over $Y$. The
trivialization of $\gamma^*\eta$ yields an isomorphism
$$
V\,=\, V\otimes {\mathcal O}_Y
\, \longrightarrow\, V\otimes \gamma^*\eta\, .
$$
This isomorphism gives an isomorphism
$$
\gamma_*V\, \longrightarrow\, \gamma_*(V\otimes \gamma^*\eta)
\,=\, (\gamma_*V)\otimes \eta
$$
(the vector bundle $(\gamma_*V)\bigotimes\eta$ is identified with
$\gamma_*(V\bigotimes \gamma^*\eta)$ using the projection formula).

Therefore, if $V\, \in\, N_\gamma$, and $\gamma_*V$ is stable,
then $\gamma_*V\, \in\, M^\eta$.

Take any $E\, \in\, {M}^\eta$. Fix an isomorphism
$$
\beta\, :\, E\, \longrightarrow\, E\otimes\eta\, .
$$
Since $E$ is stable, it follows that $E$ is simple, which means
that all global endomorphisms of $E$ are constant scalar
multiplications. Therefore, any two isomorphisms between
$E$ and $E\bigotimes\eta$ differ by multiplication
with a constant scalar.

We re-scale the section $\beta$ by multiplying it with a
nonzero scalar such that the $m$--fold composition
\begin{equation}\label{c.c}
\overbrace{\beta\circ\cdots\circ
\beta}^{m\mbox{-}\rm{times}}\,:\,
E\, \longrightarrow\, E\otimes\eta^{\otimes m}
\end{equation}
coincides with $\text{Id}_E\bigotimes\sigma$,
where $\sigma$ is the section in \eqref{sigma}.

Let
\begin{equation}\label{b}
\theta\, \in\, H^0(X,\, \mathcal{E}nd(E)\otimes\eta)
\end{equation}
be the section defined by $\beta$. Consider the pullback
$$
\gamma^* \theta\, \in\, H^0(Y,\, \gamma^* \mathcal{E}nd(E)
\otimes\gamma^*\eta)
$$
of the section in \eqref{b}.
Using the canonical trivialization of $\gamma^*\eta$,
it defines a section
\begin{equation}\label{b1}
\theta_0\, \in\, H^0(Y,\, \gamma^* \mathcal{E}nd(E))\, =\,
H^0(Y,\, \mathcal{E}nd(\gamma^* E))\, .
\end{equation}
Since $Y$ is irreducible, the
characteristic polynomial of $\theta_0(y)$ is independent of
$y\, \in\, Y$. Therefore, the eigenvalues of
$\theta_0(y)$, along with their multiplicities,
do not change as $y$ moves over $Y$.
Consequently, for each eigenvalue $\lambda$ of
$\theta_0(y)$, we have the associated generalized eigenbundle
\begin{equation}\label{eb}
\gamma^* E\, \supset\, E^\lambda\, \longrightarrow\, Y
\end{equation}
whose fiber over any $y\, \in\, Y$ is the
generalized eigenspace of $\theta_0(y)\, \in\, \text{End}
\left((\gamma^* E)_y\right)$ for the fixed eigenvalue $\lambda$.

Since the composition in \eqref{c.c} coincides with 
$\text{Id}_E\bigotimes\sigma$, we have
$$
\overbrace{\theta_0\circ\cdots\circ\theta_0}^{
m\mbox{-}\rm{times}}\, =\, \text{Id}_{\gamma^* E}\, ,
$$
where $\theta_0$ is constructed in \eqref{b1} from $\beta$.
Therefore, if $\lambda$ is an eigenvalue of $\theta_0(x)$, then
$\lambda^m\, =\,1$.

The Galois group $\text{Gal}(\gamma)$ for the
covering $\gamma$ in \eqref{e6} is identified with
the group of all $m$--th roots of $1$; the group
of $m$--th roots of $1$ will be denoted
by $\mu_m$. The action of $\mu_m$ on $Y$ is obtained by
restricting the multiplicative action of ${\mathbb G}_m$
on the line bundle $\eta$. Note that $\text{Gal}(\gamma)$
has a natural action on the pullback
$\gamma^* E$ which is a lift of the action
of $\text{Gal}(\gamma)$ on $Y$. Examining the construction
of $\theta_0$ from $\beta$ it follows that the action of any
$$
\rho\, \in\, \text{Gal}(\gamma)\,=\, \mu_m
$$
on $\gamma^* E$ takes the eigenbundle
$E^\lambda$ (see \eqref{eb}) to the eigenbundle
$E^{\lambda\cdot\rho}$. This immediately implies that each element
of $\mu_m$ is an eigenvalue of $\theta_0(y)$ (we noted
earlier that the eigenvalues lie in $\mu_m$), and the
multiplicity of each eigenvalue of $\theta_0(y)$ is $r/m$.

Consider the subbundle
\begin{equation}\label{E1}
E^1\, \longrightarrow\, Y
\end{equation}
of $\gamma^*E$, which is the eigenbundle for the eigenvalue
$1\, \in\, \mu_m$ (see \eqref{eb}). Define
$$
\widetilde{E}^1\, :=\, \bigoplus_{\rho\in \text{Gal}(\gamma)}
\rho^* E^1\, .
$$
There is a natural action of $\text{Gal}(\gamma)\,=\,
\mu_m$ on $\widetilde{E}^1$. Since the action of any $\rho
\, \in\,\mu_m$ on $\gamma^* E$ takes the eigenbundle
$E^\lambda$ to the eigenbundle
$E^{\lambda\rho}$, it follows immediately that
we have a $\text{Gal}(\gamma)$--equivariant identification
\begin{equation}\label{E2}
\gamma^* E\,=\, \widetilde{E}^1\, :=\,
\bigoplus_{\rho\in \text{Gal}(\gamma)}\rho^* E^1\, .
\end{equation}
In view of this $\text{Gal}(\gamma)$--equivariant isomorphism we 
conclude that the vector bundle $\gamma_{*}E^1$
is isomorphic to $E$.

To complete the proof of the lemma it remains to show that
the vector bundle $E^1$ is stable.

Take any vector bundle $W\,\longrightarrow\, Y$.
Since the map $\gamma$ is finite, we have
$$
H^i(Y\, ,W)\, =\, H^i(X,\, \gamma_*W)
$$
for all $i$. Let $d_W$ (respectively, $\overline{d}$) be the degree
of $W$ (respectively, $\gamma_*W$), and let $r_W$ be the rank of $W$; 
so, $\text{rank}(\gamma_*W)\,=\, mr_W$. From Riemann--Roch theorem
and the Hurwitz's formula
\begin{equation}\label{gY}
\text{genus}(Y)\,=\, m(g-1)+1
\end{equation}
we have
$$
\chi(W) \,=\, d_W- r_Wm(g-1)\,=\, \overline{d} - mr_W(g-1)
\,=\, \chi(\gamma_*W)\, .
$$
Hence
\begin{equation}\label{er1}
d_W\,=\, \overline{d}\, .
\end{equation}
We now note that if $W_1\, \subset\, W$ is a nonzero
algebraic subbundle such that
$$
\text{degree}(W_1)/\text{rank}(W_1)\, \geq\,
\text{degree}(W)/\text{rank}(W)\, ,
$$
then from \eqref{er1} we have 
$\text{degree}(\gamma_*W_1)/\text{rank}(\gamma_*W_1)\, 
\geq\, \text{degree}(\gamma_*W)/\text{rank}(\gamma_*W)$.
Hence $W$ is stable if $\gamma_*W$ is so.
In particular, $E^1$ is stable because $\gamma_* E^1
\,=\, E$ is stable.
\end{proof}

Using \eqref{gY}, we have
\begin{equation}\label{e7}
\dim {N}_\gamma\, =\, \frac{r^2}{m^2} \left(\text{genus}(Y)
-1\right)+1-g\,=\,(g-1)\left(\frac{r^2}{m}-1\right)\, .
\end{equation}
Hence $\dim M- \dim N_\gamma\, > \, 3$. Let
\begin{equation}\label{Z}
{\mathcal Z}\,:=\, \bigcup_{L\in\Gamma\setminus\{
{\mathcal O}_X\}} {M}^L\, \subset\, M
\end{equation}
be the closed subscheme. We note that Lemma \ref{lem1} has the
following corollary.

\begin{corollary}\label{cor1}
The codimension of the closed subscheme ${\mathcal Z}\,
\subset\, M$ is at least three.
\end{corollary}

For a vector space $V$, by ${\mathbb P}(V)$ we will
denote the projective space parametrizing all lines in $V$. 
Similarly, for a vector bundle $W$, by ${\mathbb P}(W)$ we
denote the projective bundle parametrizing all lines in $W$.

\begin{lemma}\label{lem2}
Take any $E\, \in\, M\setminus
{\mathcal Z}$, where $\mathcal Z$ is constructed
in \eqref{Z}. Consider the vector bundle
$F\, =\,\phi_\eta(E)\,=\, E\bigotimes\eta$,
where $\phi_\eta$ is the automorphism
in \eqref{e2}. Then there is a unique isomorphism
of the projective bundle
$$
{\mathbb P}(F)\, \longrightarrow\, X
$$
with ${\mathbb P}(E)\, \longrightarrow\, X$
over the identity map of $X$.
\end{lemma}

\begin{proof}
There is a natural isomorphism of ${\mathbb P}(E)$
with ${\mathbb P}(E\bigotimes\eta)$. If there are two
distinct isomorphisms of ${\mathbb P}(E)$
with ${\mathbb P}(E\bigotimes L)$, we get a nontrivial
automorphism of ${\mathbb P}(E)$.

Let $f\, :\, {\mathbb P}(E)\, \longrightarrow\,
{\mathbb P}(E)$ be a nontrivial
automorphism. There is a line bundle
$L_0$ and an isomorphism of vector bundles
$$
\widetilde{f}\, :\, E\, \longrightarrow\,
E\otimes L_0
$$
such that $\widetilde{f}$ induces $f$.
Since $\widetilde{f}$ 
induces an isomorphism of $\det E$ with
$\det (E\bigotimes L_0)\, =\, (\det E)\bigotimes
L^{\otimes r}_0$, we conclude that $L_0\, \in\,
\Gamma$. Next note that $L_0$ must be trivial because
$E\, \notin\, {\mathcal Z}$.

The vector bundle
$E$ is simple because it is stable. Hence all
automorphisms
of $E$ induce the identity map of ${\mathbb P}(E)$.
Therefore, ${\mathbb P}(E)$ does not admit a nontrivial
automorphism. This completes the proof of the lemma.
\end{proof}

\begin{prop}\label{prop-n1}
The morphism $\sN(r) \,\longrightarrow\, N(r)$
in \eqref{c1} is an isomorphism in codimension three.
\end{prop}

\begin{proof}
Let $i\, \in \, {\mathbb Z}/r{\mathbb Z}$ be the image
of $\text{degree}(\xi)$. In the
proof of Lemma \ref{lem2} we saw that if a projective bundle
$P\, \in\, N(r)_i$ admits a nontrivial automorphism, then
$P$ lies in the image of the subscheme ${\mathcal Z}$ by
the projection $M(r,\xi)\, \longrightarrow\, N(r)_i$. Therefore,
the proposition follows from Corollary \ref{cor1}.
\end{proof}

Let
\begin{equation}\label{U}
{\mathcal U}\, :=\, M \setminus {\mathcal Z}
\end{equation}
be the Zariski open subset, where ${\mathcal Z}$ is constructed
in \eqref{Z}. Let
\begin{equation}\label{a}
{\mathcal P}\, \longrightarrow\, M\times X
\end{equation}
be the universal projective bundle (see \cite{BBGN}).

Fix a point $x\, \in\, X$. Consider the restriction
\begin{equation}\label{P}
{\mathcal P}_x\, :=\, {\mathcal P}\vert_{M\times\{x\}}
\, \longrightarrow\, M\, .
\end{equation}
Let
\begin{equation}\label{P0}
{\mathcal P}^0\, :=\, {\mathcal P}_x\vert_{\mathcal U}
\, \longrightarrow\, {\mathcal U}
\end{equation}
be the restriction of ${\mathcal P}_x$ to the open subset
$\mathcal U$ defined in \eqref{U}.

The action of $\Gamma$ on $M$ (see \eqref{e3})
clearly preserves $\mathcal U$. The resulting action of
$\Gamma$ on $\mathcal U$ and the trivial action of $\Gamma$
on $X$ together define an action of $\Gamma$ on
$\mathcal U\times X$.

Lemma \ref{lem2} has the following corollary:

\begin{corollary}\label{cor2}
There is a unique lift of the action of $\Gamma$ on
$\mathcal U\times X$ to the projective bundle
$$
{\mathcal P}^U\, :=\,
{\mathcal P}\vert_{\mathcal U\times X}\, \longrightarrow\,
\mathcal U\times X\, ,
$$
where ${\mathcal P}$ is defined in \eqref{a}. 
In particular, the projective bundle ${\mathcal P}^0$ in
\eqref{P0} admits a canonical lift of the action of $\Gamma$
on $\mathcal U$, which is obtained by restricting the action
of $\Gamma$ on ${\mathcal P}\vert_{\mathcal U\times X}$.
\end{corollary}

\begin{remark}\label{rem1}
{\rm Take any vector bundle $E\, \in\, M^\eta$ with $\eta\, \not=\,
{\mathcal O}_X$. Then the automorphism of ${\mathbb P}(E)$
given by an isomorphism of $E$ with $E\bigotimes\eta$ is
nontrivial. Therefore, a vector bundle $F$ lies in
${\mathcal U}$ if and only if the projective bundle ${\mathbb P}(F)$ 
does not admit any nontrivial automorphism.}
\end{remark}

\section{Moduli space of $\text{PGL}(r,k)$--bundles}\label{sec3}

Define 
\begin{equation}\label{UP}
{\mathcal U}_P\, :=\, {\mathcal U}/\Gamma\, ,
\end{equation}
where ${\mathcal U}$ and $\Gamma$ are defined in \eqref{U} and
\eqref{e1} respectively. Note that $\Gamma$ acts
freely on $\mathcal U$ and thus ${\mathcal U}_P$ 
is an open substack of the component $\sN(r)_i$ of $\sN(r)$, where
$$
\deg{\xi}\,\equiv\, i\, \text{~mod~} \, r\, .
$$
In fact, ${\mathcal U}_P$ is also the moduli stack of
stable $\text{PGL}(r,k)$--bundles over $X$ without
automorphisms and with topological invariant
$\text{degree}(\xi)\, \in\, {\mathbb Z}/r{\mathbb Z}$
(see Remark \ref{rem1}).

We will now recall the descriptions of $\text{Pic}({\mathcal U}_P)$ 
and $\text{Pic}(M(r,\xi))$. Let $\delta$ be the greatest common 
divisor of $r$ and
$\text{degree}(\xi)$. Fix a semistable vector bundle 
$$
F\, \longrightarrow\, X
$$
such that $\text{rank}(F)\, =\, r/\delta$
and
$$
\text{degree}(F)\,=\,\frac{r(g-1)-\text{degree}(\xi)}{\delta}\, .
$$
Let $\widetilde{N}$ be the moduli space of semistable vector
bundles of rank $r^2/\delta$ and degree $r^2(g-1)/\delta$. This
moduli space has the theta divisor $\Theta$ that parametrizes
vector bundles $V$ with $h^0(V)\, \not=\,0$ (from Riemann--Roch, 
$h^0(V)\, =\, h^1(V)$ for all $V\,\in\, \widetilde{N}$). We 
have a morphism
$$
\psi\, :\, M(r,\xi)\, \longrightarrow\, \widetilde{N}
$$
defined by $E\, \longmapsto\, E\bigotimes F$, where $F$ is
the vector bundle fixed above.

The Picard group of $M(r,\xi)$ is isomorphic to $\mathbb Z$,
and the ample generator of $\text{Pic}(M(r,\xi))$ is the
pull back
\begin{equation}\label{pic}
{\mathcal L}_0\, :=\, \psi^*{\mathcal O}_{\widetilde{N}}(\Theta)
\,\longrightarrow\,M(r,\xi)
\end{equation}
(see \cite{DN}).

\begin{remark}\label{retda}
{\rm From Corollary \ref{cor1} it follows that the inclusion of
${\mathcal U}$ in $M(r,\xi)$ (see \eqref{U} for
${\mathcal U}$) induces an isomorphism of Picard groups.
Therefore, ${\rm Pic}({\mathcal U})\,=\, \mathbb Z$, and it
is generated by the restriction of ${\mathcal L}_0$.}
\end{remark}

\begin{lemma}\label{lem3}
Let ${\mathcal U}_P$ be the quotient defined
in \eqref{UP}. Then the image of the homomorphism
$$
{\rm Pic}({\mathcal U}_P)\,
\longrightarrow\,{\rm Pic}({\mathcal U})
$$
induced by the quotient map is generated
by ${\mathcal L}^\delta_0$, where $\delta\,=\,{\rm g.c.d.}
(r\, ,{\rm degree}(\xi))$, and ${\mathcal L}_0$ is the line
bundle in \eqref{pic}.
\end{lemma}

The above lemma follows from \cite[p. 184, Theorem]{BLS}.

The quotient
\begin{equation}\label{a0}
{\mathcal P}^U/\Gamma\, \longrightarrow\,{\mathcal U}_P\times X
\end{equation}
(see Corollary \ref{cor2}) is the universal projective bundle. Let
\begin{equation}\label{PP}
{\mathcal P}^0_P\, :=\, ({\mathcal P}^U/\Gamma)\vert_{
{\mathcal U}_P\times\{x\}}\, \longrightarrow\,{\mathcal U}_P
\end{equation}
be the projective bundle, where $x$
is a fixed point of $X$ as in \eqref{P}. Note that
$$
{\mathcal P}^0_P\, =\,{\mathcal P}^0/\Gamma\, ,
$$
where ${\mathcal P}^0$ is constructed in \eqref{P0}, and
the quotient is for the canonical action of
$\Gamma$ constructed in Corollary \ref{cor2}.

\begin{proof}[Proof of Theorem \ref{thm:br}]
Fix a line bundle $\xi\, \longrightarrow\, X$ such that 
$\deg(\xi)\,\equiv\, i \ {\rm mod} \ r$. 
Let
$$
\pi\,:\,{\mathcal U}\, \longrightarrow\,
{\mathcal U}/\Gamma
$$
be the quotient map (see \eqref{UP}). By Leray 
spectral sequence, and the identification
$$H^1(\Gamma,\, \text{Pic}({\mathcal U}))
\,=\, H^1(\Gamma,\,\Z)\,=\, 0\, ,
$$
we get the following exact sequence:
\begin{equation}\label{ec}
 \Pic({\mathcal U}_P) \, \longrightarrow\, 
\Pic({\mathcal U})^{\Gamma}
\, \stackrel{\varphi}{\longrightarrow}\,
H^2(\Gamma,k^*) \, \longrightarrow\, Br({\mathcal U}_P)\, \longrightarrow\, 
Br({\mathcal U})^{\Gamma}\, .
\end{equation}

The action of $\Gamma$ on $\Pic({\mathcal U})$ is trivial because
the action of $\Gamma$ preserves the ample generator
of $\Pic(\sM(r,\xi))\,=\, \Pic({\mathcal U})$ (see Remark
\ref{retda}). Thus by Lemma \ref{lem3}, the
cokernel of the homomorphism $$\Pic(\sN(r)_i)
\, \longrightarrow\, 
\Pic(\sM(r,\xi))^{\Gamma}$$ is isomorphic to $\Z/\delta$,
where $\delta\,=\, \text{g.c.d.}(r,i)$.

{}From Corollary \ref{cor1} it follows that the inclusion map of
${\mathcal U}$ in $M(r,\xi)$ induces an isomorphism of Brauer groups.
The Brauer group of $M(r,\xi)$ is generated by the class of
the projective bundle ${\mathcal P}^0$ constructed in
\eqref{P0} \cite{BBGN}.
Therefore, by Corollary \ref{cor2}, the homomorphism
$$
Br({\mathcal U}_P)\, \longrightarrow\, Br({\mathcal U})
$$
is surjective.

{}From Corollary \ref{cor1} we know that the inclusion map of
${\mathcal U}_P$ in ${\mathcal N}(r)_i$ induces an isomorphism
of Brauer groups (see also Proposition \ref{prop-n1}).
Thus we get the following exact sequence
\begin{equation}\label{H}
0\, \longrightarrow\, 
\frac{H^2(\Gamma,k^*)}{H} \, \longrightarrow\, Br(\sN(r)_i) 
\,\stackrel{\tau}{\longrightarrow}\, \Z/\delta
\, \longrightarrow\, 0\, ,
\end{equation}
where $H\subset H^2(\Gamma,k^*)$ is the 
subgroup of order
$\delta$. It remains to find the generator of $H$.

Consider the line bundle ${\mathcal L}_0$ in \eqref{pic}.
Let $\widetilde{\Gamma}_i$ be the group of all pairs of the form
$(\zeta\, ,\sigma)$, where $\zeta\, \in\, \Gamma$, and $\sigma$ is an
isomorphism ${\mathcal L}_0\, \longrightarrow\, \phi^*_\zeta{\mathcal 
L}_0$; the map $\phi_\zeta$ is defined in \eqref{e2}.
The group operation on $\widetilde{\Gamma}_i$ is the following:
$$
(\zeta_1\, ,\sigma_1(\phi_\zeta(z)))\cdot (\zeta\, ,\sigma(z))
\,=\,(\zeta_1\otimes\zeta\, , (\sigma_1\circ \sigma)(z))\, ,
$$
where $z\, \in\, M(r,\xi)$. There is a natural projection
$\widetilde{\Gamma}_i\, \longrightarrow\, \Gamma$ that sends
any $(\zeta\, ,\sigma)$ to $\zeta$. Consequently, we have the
central extension
$$
0\, \longrightarrow\, k^*\, \longrightarrow\, \widetilde{\Gamma}_i
\, \longrightarrow\,\Gamma \, \longrightarrow\, 0\, .
$$
Let
\begin{equation}\label{nui}
\nu_i\, \in\, H^2(\Gamma\, ,k^*)
\end{equation}
be the corresponding extension class. First note that
$$
\nu_i\, =\, \varphi({\mathcal L}_0)\, ,
$$
where $\varphi$ is the homomorphism in \eqref{ec}.
Recall that ${\mathcal L}_0$ is the generator of
$\text{Pic}({\mathcal U})$. Hence the subgroup $H$
is \eqref{H} is generated by
$\nu_i$. This completes the proof of the theorem.
\end{proof}

\begin{remark}\label{rem-ex}
{\rm Let $\bigwedge^2 \Gamma$ be the quotient of
$\Gamma\bigotimes_{\mathbb Z}\Gamma$ by the subgroup of
elements of the form $x\otimes y- y\otimes x$, where
$x\, ,y\, \in\, \Gamma$.
The space of all extensions of $\Gamma$ by $k^*$ is
parametrized by $\text{Hom}(\bigwedge^2 \Gamma\, ,k^*)$
(see \cite[pp. 217--218, Theorem 4.4]{Ra}). We have
$\Gamma\,=\, H^1_{et}(X,\, \mu_r)$, and $\text{Hom}
(\bigwedge^2 \Gamma\, ,k^*)\,=\,\text{Hom}
(\bigwedge^2 H^1_{et}(X,\, \mu_r)\, ,\mu_r)$, where $\mu_r$
is the group of $r$--th roots of $1$. The cup product
$$
H^1_{et}(X,\, \mu_r)\otimes H^1_{et}(X,\, \mu_r)
\, \longrightarrow\, H^2_{et}(X,\, \mu_r)\,=\, \mu_r
$$
defines an element $\widehat{\nu}\,\in\, \text{Hom}
(\bigwedge^2 H^1_{et}(X,\, \mu_r)\, ,\mu_r)$. The element
$\nu_i$ in \eqref{nui} coincides with
$(r/\delta)\cdot\widehat{\nu}$.}
\end{remark}

\section{Twisted bundles on a $\mu_r$--gerbe over a curve}\label{sec4}

In this section we prove some results on twisted sheaves on a 
$\mu_r$--gerbe over a curve; these will be required in the proof of 
Theorem \ref{thm:order=r}. We fix the following notation:
\begin{enumerate}
\item[$\bullet$] $K$ is any field of characteristic zero, and 
$\overline{K}$ is an algebraic closure of $K$.
\item[$\bullet$] $X/K$ is a smooth projective geometrically connected 
curve of genus $g\,\geq\, 2$, and having a $K$--point.
\item[$\bullet$] $\sL$ and $\Lambda$ are line bundles on $X$, with 
${\rm degree}(\sL)\,=\,d$ and ${\rm degree}(\Lambda)\,=\,1$.
\item[$\bullet$] $h\,:\,Y\,\longrightarrow\, X$ is a $\mu_r$--gerbe;
if $g\,=\,2$, then we assume that $r\, \geq\, 3$.
\item[$\bullet$] $\sT$ is the moduli stack of stable twisted sheaves on 
$X$ of rank $r$ and determinant $\sL$, and $q\,:\,\sT\,\longrightarrow\, 
T$ is its coarse moduli space.
\item[$\bullet$] $\E\,\longrightarrow\, \sT\times_KY$ is the universal 
twisted bundle.
\item[$\bullet$] $\overline{\sT}\,:=\, \sT\times_K\overline{K}$, and 
$\overline{T}\, :=\,T\times_K\overline{K}$.
\end{enumerate}

Let $\xi$ be a line bundle on $X$
of degree $d$. As before, $\sM(r,\xi)$ will
denote the moduli stack of vector bundles over $X$ of
rank $r$ and determinant $\xi$. 
Let $\pi\,:\,\sM(r,\xi)\times X\, 
\longrightarrow\,\sM(r,\xi)$ denote the natural projection. Let 
$$
\delta\,=\,{\rm g.c.d.}(r,d)\,=\, {\rm g.c.d.}(r,\chi)\, ,
$$
where $\chi\,=\, \chi(\xi)- (r-1)(g-1)$. So $\chi\,=\, \chi(E)$
for any vector bundle $E\,\longrightarrow\, X$ parametrized by 
$\sM(r,\xi)$. Let
$$
{\mathcal E}\, \longrightarrow\, \sM(r,\xi)\times X
$$
be the universal vector bundle. Define the line bundle
\begin{equation}\label{aF}
F \,:=\, 
\det(\pi_*\mathcal E)\otimes 
\det(R^1\pi_*{\mathcal E})^{*}\, \longrightarrow\,\sM(r,\xi)
\end{equation}
be the line bundle.
Let $i_x\,:\,\sM(r,\xi)\,\longrightarrow\, \sM(r,\xi)\times X$ be
the constant section determined by a $K$--point $x\,\in\, X(K)$.

See \cite{DN} for the following theorem.

\begin{theorem}\label{generator}
The line bundle $(F^*)^{\otimes \frac{r}{\delta}}\otimes
(i_x^*\det({\mathcal E}))^{\otimes \frac{\chi}{\delta}}$
descends to a line bundle on the coarse moduli space $M(r,\xi)$,
where $F$ is constructed in \eqref{aF}. Moreover, this
descended line bundle is
the ample generator of $\Pic(M(r,\xi))\,=\, {\mathbb Z}$.
\end{theorem}

\begin{prop}\label{prop-l}
Using the above notation:
\begin{enumerate}
\item[(i)]$\overline{\sT}$ is non--canonically isomorphic to the moduli 
stack of stable vector bundles on $X\times_K\overline{K}$ of rank $r$ 
and 
determinant $\zeta$ for some fixed line bundle $\zeta$. Further, if 
$$Y\times_K\overline{K}
\, \longrightarrow\, X\times_K\overline{K}$$ is a neutral 
$\mu_r$--gerbe, 
then ${\rm degree}(\zeta)\,\equiv\, d \ {\rm mod}\ r$ for any such 
$\zeta$.
\item[(ii)]$\Pic(T)\,=\,\Pic(\overline{T})\,=\,\Z$. 
\item[(iii)] The natural homomorphism $Br(K)\,\longrightarrow\, Br(T)$ 
is injective.
\end{enumerate}
\end{prop}

\begin{proof}
$(i)$: This is proved in \cite[3.1.2.1]{lieblich}.

$(ii)$: This follows from statement (i) and Theorem \ref{generator}.

$(iii)$: This follows immediately from statement (ii) and the 
exactness of the following sequence:
$$ 0\,\longrightarrow\, \Pic(T)\,\longrightarrow\, 
\Pic(\overline{T})^{\Gal(K)}\,\longrightarrow\,
Br(K) \,\longrightarrow\, Br(T)\, .$$
This completes the proof of the proposition.
\end{proof}

\begin{prop}\label{prop:ordern}
Assume the $\mu_r$--gerbe $Y\,\longrightarrow\, X$ is a pull back of a 
$\mu_r$--gerbe on $\Spec(K)$ which
defines a class of order $r$ in $Br(K)$. For any 
$K$--point $x$ in $X$, let $$i_x\,:\, T\,\longrightarrow\, T\times_KX$$ 
denote the corresponding section. Then the following two
statements hold.
\begin{enumerate}
\item[(i)]The class in $Br(T)$ defined by the $\G_m$--gerbe 
$\sT\,\longrightarrow\, 
T$ has order exactly equal to $\delta\,=\, {\rm g.c.d.}(r,d)$. 
\item[(ii)]The order of $i_x^*({\sE}nd(\E))$ is exactly equal to $r$. 
\end{enumerate}
\end{prop}

\begin{proof}
Let $\gamma$ denote the class of $\sT\,\longrightarrow\, T$ in $Br(T)$. 
Let us first show 
that $\gamma$ is annihilated by $r$. Consider the Azumaya algebra ${\sE} 
nd(\E)$ on $T\times_KX$. The class of $i_x^*({\sE}nd(\E))$ in $Br(T)$ 
is 
$\gamma+\beta$ where $\beta\,\in\, Br(T)$ is the image,
by the natural homomorphism $Br(K)\,\longrightarrow\, Br(\sT)$,
of the class in $Br(K)$ defined by the restriction of 
$Y\,\longrightarrow\, X$ 
to $x$. We have $r\cdot 
\beta\,=\,0$ because $\beta$ is associated to
a $\mu_r$--gerbe. Moreover, $r\cdot(\gamma+\beta)\,=\,0$ since 
$\gamma+\beta$ is represented by an Azumaya algebra of rank $r^2$. Thus 
$r\cdot \gamma\,=\,0$.

Proof of $(i)$: Note that after base field extension to $\overline{K}$, 
the 
image of 
$\gamma$ in $Br(\overline{T})$ is precisely $\delta$. Thus it is 
sufficient 
to show $\delta\cdot \gamma\,=\,0$. There exists an integer $m$ and a 
twisted 
vector bundle $\F$ on $Y$ of rank $r^m$ and trivial determinant. 
Consider the vector bundle $$\sH \,=\,\E \boxtimes \F^{*}$$ on 
$\sT\times Y$. This vector bundle descends to a vector bundle on 
$\sT\times X$. Also note that the class of $i_x^*{\sE} nd(\sH)$ in 
$Br(T)$ 
is precisely $\gamma$. Let $\pi\,:\,\sT\times X \,\longrightarrow\, \sT$ 
denote the natural 
projection. $\sH$ can be thought of as a family of vector bundles on 
$X$ parametrized by $\sT$. Since this family is bounded, there exists 
an integer $m_0$ such that for all $m\, \geq\, m_0$,
$$
\sH\boxtimes \Lambda^{m}
$$
has no higher cohomology on fibers of $\pi$. The vector bundle
$$ \pi_*\left(\sH\boxtimes\Lambda^{m}\right)$$
is a twisted vector bundle on $\sT$ of rank
$$ d+rm+(r^m+r)\cdot(1-g)\, .$$
Thus $\gamma$ is annihilated by $d+rm+(r^m+r)\cdot(1-g)$. We already 
know that it is annihilated by $r$. This proves $\delta\cdot 
\gamma\,=\,0$.

Proof of $(ii)$: Let $a(\gamma+\beta)\,=\,0$. We will show $a$ is 
divisible by $r$. 
Since the image of $\gamma$ in $Br(\overline{T})$ has order $\delta$ and 
image of $\beta$ in $Br(\overline{T})$ is zero, it follows that $a$ is 
divisible by $\delta$. Thus we get
\begin{equation}\label{x}
a\cdot \beta \,=\, 0\, .
\end{equation}
By Proposition \ref{prop-l}(iii), the order of the class $\beta$
is $r$. Hence \eqref{x} implies that $a$ is divisible by $r$.
This completes the proof of the proposition.
\end{proof}

\section{Proof of Theorem \ref{thm:order=r}}\label{sec5}

\begin{proof}[Proof of $(i)$:] Let
${\mathcal E}$ be the universal vector bundle over
${\mathcal M}(r,\xi)\times X$. For any point $x\, \in\, X$,
let ${\mathcal E}_x\, \longrightarrow\, {\mathcal M}(r,\xi)$
be the vector bundle obtained by restricting ${\mathcal E}$ 
to ${\mathcal M}(r,\xi)\times \{x\}$.
Take two points $x\, ,y\, \in\, X$. Note that the multiplicative
action of $k^*$ on ${\mathcal E}$ induces a trivial action
of $k^*$ on ${\mathcal E}^*_x\bigotimes {\mathcal E}_y$.
Therefore, ${\mathcal E}^*_x\bigotimes {\mathcal E}_y$ descends to
${M}(r,\xi)$. This descended vector bundle on
${M}(r,\xi)$ will also be denoted by ${\mathcal E}^*_x\bigotimes 
{\mathcal E}_y$. Next we note that the action of
$\Gamma$ on ${\mathcal M}(r,\xi)$ lifts to an action of $\Gamma$ on
the vector bundle ${\mathcal E}^*_x\bigotimes {\mathcal E}_y$. Let
${\mathcal W}$ denote the vector bundle on ${\mathcal U}_P$
defined by the $\Gamma$--equivariant vector bundle ${\mathcal 
E}^*_x\bigotimes {\mathcal E}_y$. It is easy to see
that the $\text{PGL}(r^3,k)$--bundle ${\mathcal P}^0_P\bigotimes
{\mathbb P}({\mathcal W})$, where ${\mathcal P}^0_P$ is constructed
in \eqref{PP}, is isomorphic to ${\mathcal P}^{y,0}_P\bigotimes
{\mathbb P}(\mathcal{E}nd({\mathcal E}_x))$, where ${\mathcal 
P}^{y,0}_P$
is the projective bundle obtained by substituting $x$ with $y$
in the construction of ${\mathcal P}^0_P$. Consequently,
the class $\alpha$ in Theorem \ref{thm:order=r} is independent
of $x$. From \cite{BBGN} we know that $\alpha$ maps to
the generator of $Br({\mathcal M}(r,\xi))$. Hence $\alpha$
maps to $1\,\in\, \Z/\delta\Z$ in Theorem \ref{thm:br}.

Proof of $(ii)$: Let $\sN(r)_i$ be a connected component of $\sN(r)$ for 
$i\in\Z/r$. To prove the theorem it is enough to construct a $k$--scheme 
$T$ together with an Azumaya algebra $\sB$ on $T\times_kX$ such that 
\begin{enumerate}
 \item[(i)] $\sB$ gives a family of stable ${\rm PGL}(r,k)$ bundles on 
$X$ 
lying in component $\sN(r)_i$, and 
 \item[(ii)] if $i_x\,:\,T\,\longrightarrow\, T\times_k X$ is the 
section given 
by the point $x$, then $i_x^*\sB$ has order precisely $r$ in $Br(T)$. 
\end{enumerate}

We carry out this construction below.

We first claim that there is a field extension
$K/k$ such that $Br(K)$ contains an element
$\beta$ of order $r$. To prove this, take the purely transcendental 
extension
$K\,=\,k(x,y)$, and define $\beta$ to be the class of the cyclic algebra 
$(x,y)_{\zeta}$,
where $\zeta$ is a primitive $r$--th root of unity.
Note that $(x,y)_{\zeta}$ is a division algebra.

In an earlier version, we had a very long argument for the
existence of $K$ and $\beta$. The above short argument was
provided by the referee. 

Define $X_K\,:=\, X\times_kK$. 
Let $\beta'$ denote the pullback of $\beta$ in $Br(X_K)$. Since
$\beta'$ is of order $r$, there is $\mu_r$--gerbe
$$Y\,\longrightarrow\, X_K$$
representing the class $\beta'$. We fix the following
notation:
\begin{enumerate}
 \item[$\bullet$]$\sT$ be the moduli stack of stable twisted rank $r$ 
vector bundles on $Y$ with determinant $\sL$, where  
$\deg(\sL)\,\equiv\, i ~{\rm mod}~ r$.
\item[$\bullet$] Let $\E$ be the universal bundle on $\sT\times_KY$.
\item[$\bullet$] Let $\sT\,\longrightarrow\, T$ be the coarse moduli 
space of $\sT$. 
\end{enumerate}

Now ${\sE} nd(\E)$ descends to an Azumaya algebra on $T\times_kX$. Since 
this is a stable family, we get a map $T\,\longrightarrow \,\sN(r)_i$.
The proof now follows from Proposition \ref{prop:ordern}, since order of
$i_x^*{\sE} nd(\E)$ in $Br(T)$ is precisely $r$. This completes the
proof of the theorem.
\end{proof}

Let $S$ be a smooth variety defined over $k$, with $\dim S\, \geq\, 
1$. Let
$$
{\mathbb P}_S\, \longrightarrow\, S
$$
be a projective bundle of relative dimension $r-1$. So
${\mathbb P}_S$ defines an algebraic principal
$\text{PGL}(r,k)$--bundle
\begin{equation}\label{a-1}
E_{\text{PGL}(r,k)}\, \longrightarrow\, S\, .
\end{equation}
Let
\begin{equation}\label{a-beta}
\beta\, \in\, Br(S)
\end{equation}
be the class defined by ${\mathbb P}_S$.
If the order of $\beta$ is $r$, then $E_{\text{PGL}(r,k)}$
does not admit any reduction of structure group to any proper
parabolic subgroup of $\text{PGL}(r,k)$ over any
nonempty open subset of $S$ (see \cite[p. 267, Lemma 2.1]{BBGN}).

Suppose there is a an irreducible normal projective
variety $\overline{S}$ over $k$ such that $S$ is the smooth locus of
$\overline{S}$. Fix a polarization on $\overline{S}$. Assume
that the order of $\beta$ in \eqref{a-beta} is $r$. Since
$E_{\text{PGL}(r,k)}$
does not admit any reduction of structure group to any proper
parabolic subgroup of $\text{PGL}(r,k)$ over any
nonempty open subset of $S$, the principal 
$\text{PGL}(r,k)$--bundle in \eqref{a-1} is stable.

Consider $N^0(r)$ defined in \eqref{nu0}.
Let
$$
F_{\text{PGL}(r,k)}\, \longrightarrow\, N^0(r)
$$
be the principal $\text{PGL}(r,k)$--bundle
obtained by pulling back the universal projective
bundle using the map $i_x$ in Theorem \ref{thm:order=r}
(recall that $N^0(r)$ is isomorphic to ${\sN}^0(r)$). From
the second part of Theorem \ref{thm:order=r}
(combined with Proposition \ref{thm:cod2}) we have
the following corollary:

\begin{corollary}\label{corf}
The principal ${\rm PGL}(r,k)$--bundle $F_{{\rm PGL}(r,k)}\,
\longrightarrow\, N^0(r)$ is stable.
\end{corollary}


\end{document}